\documentclass[11pt,leqno]{article}
\usepackage[margin=1in]{geometry} 
\usepackage{amssymb,amsfonts,amsmath,bbm,mathrsfs,stmaryrd,mathtools}
\usepackage{xcolor}
\usepackage{url}

\usepackage{extarrows}

\usepackage[shortlabels]{enumitem}
\usepackage{tensor}

\usepackage{xr}
\usepackage[T1]{fontenc}

\usepackage[colorlinks,
linkcolor=black!75!red,
citecolor=blue,
pdftitle={},
pdfproducer={pdfLaTeX},
pdfpagemode=None,
bookmarksopen=true,
bookmarksnumbered=true,
backref=page]{hyperref}

\usepackage{tikz}
\usetikzlibrary{arrows,calc,decorations.pathreplacing,decorations.markings,intersections,shapes.geometric,through,fit,shapes.symbols,positioning,decorations.pathmorphing}

\usepackage{braket}

\usepackage[amsmath,thmmarks,hyperref]{ntheorem}
\usepackage{cleveref}

\creflabelformat{enumi}{#2#1#3}

\crefname{section}{Section}{Sections}
\crefformat{section}{#2Section~#1#3} 
\Crefformat{section}{#2Section~#1#3} 

\crefname{subsection}{\S}{\S\S}
\AtBeginDocument{%
  \crefformat{subsection}{#2\S#1#3}%
  \Crefformat{subsection}{#2\S#1#3}% 
}

\crefname{subsubsection}{\S}{\S\S}
\AtBeginDocument{%
  \crefformat{subsubsection}{#2\S#1#3}%
  \Crefformat{subsubsection}{#2\S#1#3}% 
}

\theoremstyle{plain}

\newtheorem{lemma}{Lemma}[section]

\newtheorem{corollary}[lemma]{Corollary}
\newtheorem{theorem}[lemma]{Theorem}

\theoremstyle{plain}
\theoremnumbering{Alph}
\newtheorem{theoremN}{Theorem}

\theoremstyle{plain}
\theorembodyfont{\upshape}
\theoremsymbol{\ensuremath{\blacklozenge}}

\newtheorem{definition}[lemma]{Definition}

\newtheorem{remark}[lemma]{Remark}

\crefname{definition}{definition}{definitions}
\crefformat{definition}{#2definition~#1#3} 
\Crefformat{definition}{#2Definition~#1#3} 

\crefname{ex}{example}{examples}
\crefformat{example}{#2example~#1#3} 
\Crefformat{example}{#2Example~#1#3} 

\crefname{exs}{example}{examples}
\crefformat{examples}{#2example~#1#3} 
\Crefformat{examples}{#2Example~#1#3} 

\crefname{remark}{remark}{remarks}
\crefformat{remark}{#2remark~#1#3} 
\Crefformat{remark}{#2Remark~#1#3} 

\crefname{remarks}{remark}{remarks}
\crefformat{remarks}{#2remark~#1#3} 
\Crefformat{remarks}{#2Remark~#1#3} 

\crefname{convention}{convention}{conventions}
\crefformat{convention}{#2convention~#1#3} 
\Crefformat{convention}{#2Convention~#1#3} 

\crefname{notation}{notation}{notations}
\crefformat{notation}{#2notation~#1#3} 
\Crefformat{notation}{#2Notation~#1#3} 

\crefname{table}{table}{tables}
\crefformat{table}{#2table~#1#3} 
\Crefformat{table}{#2Table~#1#3}

\crefname{lemma}{lemma}{lemmas}
\crefformat{lemma}{#2lemma~#1#3} 
\Crefformat{lemma}{#2Lemma~#1#3} 

\crefname{proposition}{proposition}{propositions}
\crefformat{proposition}{#2proposition~#1#3} 
\Crefformat{proposition}{#2Proposition~#1#3} 

\crefname{propositionN}{proposition}{propositions}
\crefformat{propositionN}{#2proposition~#1#3} 
\Crefformat{propositionN}{#2Proposition~#1#3} 

\crefname{corollary}{corollary}{corollaries}
\crefformat{corollary}{#2corollary~#1#3} 
\Crefformat{corollary}{#2Corollary~#1#3} 

\crefname{corollaryN}{corollary}{corollaries}
\crefformat{corollaryN}{#2corollary~#1#3} 
\Crefformat{corollaryN}{#2Corollary~#1#3} 

\crefname{theorem}{theorem}{theorems}
\crefformat{theorem}{#2theorem~#1#3} 
\Crefformat{theorem}{#2Theorem~#1#3} 

\crefname{theoremN}{theorem}{theorems}
\crefformat{theoremN}{#2theorem~#1#3} 
\Crefformat{theoremN}{#2Theorem~#1#3} 

\crefname{enumi}{}{}
\crefformat{enumi}{#2#1#3}
\Crefformat{enumi}{#2#1#3}

\crefname{assumption}{assumption}{Assumptions}
\crefformat{assumption}{#2assumption~#1#3} 
\Crefformat{assumption}{#2Assumption~#1#3} 

\crefname{construction}{construction}{Constructions}
\crefformat{construction}{#2construction~#1#3} 
\Crefformat{construction}{#2Construction~#1#3} 

\crefname{equation}{}{}
\crefformat{equation}{(#2#1#3)} 
\Crefformat{equation}{(#2#1#3)}

\numberwithin{equation}{section}

\theoremstyle{nonumberplain}
\theoremsymbol{\ensuremath{\blacksquare}}

\newtheorem{proof}{Proof}
\newcommand\pf[1]{\newtheorem{#1}{Proof of \Cref{#1}}}

\newcommand\bA{{\mathbb A}}

\newcommand\bG{{\mathbb G}}
\newcommand\bH{{\mathbb H}}

\newcommand\bJ{{\mathbb J}}
\newcommand\bK{{\mathbb K}}

\newcommand\bU{{\mathbb U}}

\newcommand\cO{{\mathcal O}}

\newcommand\fu{{\mathfrak u}}

\DeclareMathOperator{\Ad}{Ad}

\newcommand{\cat}[1]{\textsc{#1}}

\newcommand{\qedhere}{\mbox{}\hfill\ensuremath{\blacksquare}}

\renewcommand{\square}{\mathrel{\Box}}

\title{Deformations of nearby subgroups and approximate Jordan constants}
\author{Alexandru Chirvasitu}

\begin{document}

\date{}

\newcommand{\Addresses}{{% additional braces for segregating \footnotesize
  \bigskip
  \footnotesize

  \textsc{Department of Mathematics, University at Buffalo}
  \par\nopagebreak
  \textsc{Buffalo, NY 14260-2900, USA}  
  \par\nopagebreak
  \textit{E-mail address}: \texttt{achirvas@buffalo.edu}

  % % \medskip
  % % 
  % % \textsc{Department of Mathematics, INSTITUTION}
  % % \par\nopagebreak
  % % \textsc{ADDRESS}
  % % \par\nopagebreak
  % % \textit{E-mail address}: \texttt{??}
  % % 

}}

\maketitle

\begin{abstract}
  Let $\mathbb{U}$ be a Banach Lie group and $S\subseteq \mathbb{U}$ an ad-bounded subset thereof, in the sense that there is a uniform bound on the adjoint operators induced by elements of $S$ on the Lie algebra of $\mathbb{U}$. We prove that (1) $S$-valued continuous maps from compact groups to $\mathbb{U}$ sufficiently close to being morphisms are uniformly close to morphisms; and (2) for any Lie subgroup $\mathbb{G}\le \mathbb{U}$ there is an identity neighborhood $U\ni 1\in \mathbb{U}$ so that $\mathbb{G}\cdot U\cap S$-valued morphisms (embeddings) from compact groups into $\mathbb{U}$ are close to morphisms (respectively embeddings) into $\mathbb{G}$.

  This recovers and generalizes results of Turing's to the effect that (a) Lie groups arbitrarily approximable by finite subgroups have abelian identity component and (b) if a Lie group is approximable in this fashion and has a faithful $d$-dimensional representation then it is also so approximable by finite groups with the same property. Another consequence is a strengthening of a prior result stating that finite subgroups in a Banach Lie group sufficiently close to a given compact subgroup thereof admit a finite upper bound on the smallest indices of their normal abelian subgroups (an approximate version of Jordan's theorem on finite subgroups of linear groups). 
\end{abstract}

\noindent {\em Key words:
  Banach Lie algebra;
  Banach Lie group;
  Hausdorff metric;
  Jordan group;
  Lie subgroup;
  ad-bounded;
  approximate morphism;
  exponential map

}

\vspace{.5cm}


%\tableofcontents

%%%%%%%%%%%%%%%%%%%%%%%%%%%%%%%%%%%%%%%%%%%%%%%%%%%%%%%%%%%%%%%%%%%%%%%%%%%%%
%%%%%%%%%%%%%%%%%%%%%%%%%%%%%%%%%%%%%%%%%%%%%%%%%%%%%%%%%%%%%%%%%%%%%%%%%%%%%
\section*{Introduction}

The {\it Banach Lie groups} of the present note are those of \cite[Definition IV.1]{neeb-inf}, \cite[\S III.1.1, Definition 1]{bourb_lie_1-3}, \cite[\S VI.5]{lang-fund}, and any number of other sources: possibly infinite-dimensional Banach manifolds as well as topological groups, with the group structure maps (multiplication and inverse) analytic.  

One motivating theme for the sequel is that of \emph{stability} in Banach Lie groups $\bU$, in the Hyers-Ulam-Rassias sense (\cite{zbMATH00149467}, \cite[pp.1-2]{jung_hur-stab} and the latter's extensive references, etc.): phenomena to the effect that (under appropriate hypotheses) almost morphisms into $\bU$ are close to morphisms, subgroups of $\bU$ close to given Lie subgroups $\bG\le \bU$ are close to subgroups of $\bG$, etc. A bit of terminology (expanding slightly on \cite[Definition 4.1]{2212.06255v1}) will help state the results.

\begin{definition}\label{def:adbdd}
  A subset $S\subseteq \bU$ of a Banach Lie group is {\it ad-bounded} (on a subspace $V\le \fu:=Lie(\bU)$) if its image
  \begin{equation*}
    \Ad(S)\subset GL(\fu)
    ,\quad
    \bU
    \xrightarrow{\Ad:=\text{{\it adjoint representation} \cite[\S III.3.12]{bourb_lie_1-3}}}
    GL(\fu)
  \end{equation*}
  is bounded (respectively has bounded restriction to $V$). 
\end{definition}

That in hand, an amalgam of the stability-flavored \Cref{th:almostmor,th:mor.almost2gp} reads:

\begin{theoremN}\label{thn:close.to}
  Let $\bU$ be a Banach Lie group and $S\subseteq \bU$ an ad-bounded subset.
  \begin{enumerate}[(1),wide]
  \item For any $\varepsilon>0$ there is some $\delta>0$ such that {\it $\delta$-morphisms}
    \begin{equation*}
      \text{compact}\quad\bK
      \xrightarrow{\quad\text{continuous}\quad}
      S
      \subseteq
      \bU
    \end{equation*}
    in the sense that
    \begin{equation}\label{eq:varphideltaclose}
      d(\varphi(st),\ \varphi(s)\varphi(t))<\delta
      ,\quad
      \forall s,t\in \bK
    \end{equation}
    are $\varepsilon$-close in the space $\cat{Cont}(\bK\to \bU)$ of continuous maps to morphisms.

  \item Given a Lie subgroup $\bG\le \bU$, for every $\varepsilon>0$ there is an origin neighborhood $U\ni 1\in \bU$ so that
  \begin{equation*}
    \forall \text{compact}\quad\bK
    \xrightarrow[\quad\text{morphism}\quad]{\quad\varphi\quad}
    \bU
    ,\quad
    % \bigcup_{s\in \bG}s\cdot \varphi(\bK)\cdot s^{-1}\subset V\cap S
    \varphi(\bK)\subset (\bG\cdot U)\cap S
  \end{equation*}
  the morphism (embedding) $\varphi$ is $\varepsilon$-close to a morphism (respectively embedding) $\bK\to \bG$.  \qedhere
  \end{enumerate}  
\end{theoremN}

The arguments build on and extend earlier work in \cite{2212.06255v1,2401.14929v2}, and also recover some of their precursors in \cite{MR1503391}:
\begin{itemize}[wide]
\item\cite[Theorem 1]{MR1503391} (on whose proof \cite[Theorem 2.19]{2212.06255v1} relies), proving that whenever a Lie group with a faithful $d$-dimensional representation is arbitrarily approximable by finite almost-subgroups, then the finite groups in question can also be assumed embeddable in $GL_d$;

\item and \cite[Theorem 2]{MR1503391}, characterizing Lie groups finitely approximable in the above-mentioned fashion: they must have torus identity components. 
\end{itemize}
\Cref{re:recearlier} and the discussion preceding \Cref{cor:almostjordbdd} notes how the two follow from \Cref{cor:haveembquot,cor:almostjordbdd}. The preceding discussion dovetails with the other motivating factor: the celebrated {\it Jordan theorem} (\cite[Theorem 2]{MR3830471}, \cite[\S 1, statement I on p.281]{zbMATH03223755} \cite[Theorem 8.29]{ragh}, etc.) stating, in one version, that for a finite-dimensional connected Lie group there is a uniform bound on the index of an abelian normal subgroup of a finite subgroup. The result features in several ways in the sources mentioned above:
\begin{enumerate}[(a),wide]
\item for one thing, \cite[Theorem 2]{MR1503391} relies on Jordan's theorem;

\item\label{item:jord.2.19} for another, the aforementioned \cite[Theorem 2.19]{2212.06255v1} is an approximate version thereof: for every compact (hence Lie) subgroup $\bG\le \bU$ of a Banach Lie group, finite subgroups of $\bU$ sufficiently close to $\bG$ have abelian normal subgroups of index $\le N=N(\bG\le \bU)<\infty$. 
\end{enumerate}

\cite[Definition 2.1]{zbMATH05994630} (also \cite[Definition 1]{MR3830471}) abstracts away from the connected Lie framework in defining a group $\bG$ to be {\it Jordan} if
\begin{equation}\label{eq:jg}
  J_{\bG}
  :=
  \sup_{\substack{\bH\le \bG\\|\bH|<\infty}}
  \inf_{\substack{\bA\trianglelefteq \bH\\\bA\text{ abelian}}}
  |\bH/\bA|<\infty
  \quad
  \left(\text{the \emph{Jordan constant} of $\bG$}\right).
\end{equation}
\Cref{cor:th.2.19.noncpct,cor:th.2.19.cpct} are quantitative approximate Jordan-type results, the latter of which also recovers and strengthens \cite[Theorem 2.19]{2212.06255v1}.

\begin{theoremN}\label{thn:jrd}
  Let $\bG\le \bU$ be a Lie subgroup of a Banach Lie group.

  \begin{enumerate}[(1),wide]
  \item If $S\subseteq \bU$ is an ad-bounded subset, for some origin neighborhood $U\ni 1\in \bU$ we have
    \begin{equation*}
      \sup_{\substack{\bH\le \bU\\|\bH|<\infty\\\bH\subset \bG\cdot U\cap S}}
      \inf_{\substack{\bA\trianglelefteq \bH\\\bA\text{ abelian}}}
      |\bH/\bA|
      \le
      J_{\bG}.
    \end{equation*}
    
  \item If $\bG$ is compact, 
    \begin{equation*}
      \inf_{\text{nbhd }U\ni 1}
      \sup_{\substack{\bH\le \bU\\|\bH|<\infty\\\bH\subset \bG\cdot U}}
      \inf_{\substack{\bA\trianglelefteq \bH\\\bA\text{ abelian}}}
      |\bH/\bA|
      \le
      J_{\bG}.
    \end{equation*}
  \end{enumerate}
  \qedhere
\end{theoremN}

% % \Cref{def:reljg} below introduces approximate variants of this invariant attached to embeddings into ambient topological groups, appropriate for the type of result \Cref{item:jord.2.19} above recalls.
% % 

% % %%%%%%%%%%%%%%%%%%%%%%%%%%%%%%%%%%%%%%%%%%%%%%%%%%%%%%%%%%%%%%%%%%%%%%%%%%%%%
% % \subsection*{Acknowledgements}
% % 

% % %%%%%%%%%%%%%%%%%%%%%%%%%%%%%%%%%%%%%%%%%%%%%%%%%%%%%%%%%%%%%%%%%%%%%%%%%%%%%
% % %%%%%%%%%%%%%%%%%%%%%%%%%%%%%%%%%%%%%%%%%%%%%%%%%%%%%%%%%%%%%%%%%%%%%%%%%%%%%
% % \section{Preliminaries}\label{se:prel}
% %

%%%%%%%%%%%%%%%%%%%%%%%%%%%%%%%%%%%%%%%%%%%%%%%%%%%%%%%%%%%%%%%%%%%%%%%%%%%%%
%%%%%%%%%%%%%%%%%%%%%%%%%%%%%%%%%%%%%%%%%%%%%%%%%%%%%%%%%%%%%%%%%%%%%%%%%%%%%
\section{Almost embeddings and relative Jordan properties}\label{se:main}

We freely refer to distance estimates in both a Banach Lie group $\bU$ and Lie algebra $\fu:=Lie(\bU)$: the former is completely metrizable \cite[\S III.1.1, Proposition 1]{bourb_lie_1-3} (the metric can always be chosen left or right-invariant if desired \cite[\S IX.3.1, Proposition 2]{bourb_top_en_2}) and the latter admits \cite[\S III.4.3, Theorem 4 and \S III.6.4, Theorem 4]{bourb_lie_1-3} a Banach norm $\|\cdot\|$ compatible with the Lie-algebra structure in the sense that
\begin{equation*}
  \|[x,y]\|\le \|x\|\cdot \|y\|
  ,\quad\forall x,y\in \fu.
\end{equation*}
Recall also (\cite[\S III.4.3, Theorem 4 and \S III.6.4, Theorem 4]{bourb_lie_1-3}, \cite[\S IV]{neeb-inf}, etc.) that small neighborhoods of $0\in \fu$ and $1\in \bU$ are analytically isomorphic via the {\it exponential map}
\begin{equation*}
  \fu
  \xrightarrow{\quad\exp\text{ or }e^{\cdot}\quad}
  \bU.
\end{equation*}

As we several times have to consider locally compact subgroups $\bG\le \bU$ of Banach Lie groups (i.e. locally compact in the inherited subspace topology), we remind \cite[Theorem IV.3.16]{neeb-lc} the reader that all such are finite-dimensional {\it Lie subgroups} in the strong sense of \cite[\S III.1.3, Definition 3]{bourb_lie_1-3} (and \cite[Remark IV.4(c)]{neeb-inf}):
\begin{itemize}[wide]
\item automatically closed;

\item and finite-dimensional Lie groups in their own right;

\item with the embedding identifying the Lie subalgebra $Lie(\bG)\le Lie(\bU)$ as a {\it complemented} \cite[\S 4.9]{nb_tvs} subspace of the ambient Banach space (finite-dimensional Banach subspaces being complemented \cite[Theorem 7.3.5]{nb_tvs}). 
\end{itemize}

We first record the following variant of \cite[Theorem 2.1]{2401.14929v2}. 

\begin{theorem}\label{th:almostmor}
  Let $\bU$ be a Banach Lie group and $S\subseteq \bU$ an ad-bounded subset. For any $\varepsilon>0$ there is some $\delta>0$ such that {\it $\delta$-morphisms}
  \begin{equation*}
    \text{compact}\quad\bK
    \xrightarrow{\quad\text{continuous}\quad}
    S
    \subseteq
    \bU
  \end{equation*}
  in the sense of \Cref{eq:varphideltaclose} are $\varepsilon$-close in the space $\cat{Cont}(\bK\to \bU)$ of continuous maps to morphisms.
\end{theorem}
\begin{proof}
  The already-mentioned \cite[Theorem 2.1]{2401.14929v2} proves essentially the claim made here, except the domain compact group $\bK$ is fixed. That proof proceeds by deforming $\delta$-morphisms analytically into genuine morphisms by averaging almost-cocycles against the Haar probability measure of the domain group $\bK$. The closeness estimates depend only on a uniform bound on $\Ad s$, $s\in \bK$ and the $\delta$ of \Cref{eq:varphideltaclose}, so go through uniformly for arbitrary domains $\bK$ so long as $\delta$ and $S$ are kept fixed.
\end{proof}

A follow-up and consequence of \Cref{th:almostmor} concerns a slightly different setup: given genuine morphisms that ``just barely miss'' a Lie subgroup of a Banach Lie group, one might want those deformed into morphisms taking valued in said subgroup. 

\begin{theorem}\label{th:mor.almost2gp}
  Let $\bG\le \bU$ be a Lie subgroup of a Banach Lie group and $S\subseteq \bU$ an ad-bounded subset.

  For every $\varepsilon>0$ there is an origin neighborhood $U\ni 1\in \bU$ so that
  \begin{equation*}
    \forall \text{compact}\quad\bK
    \xrightarrow[\quad\text{morphism}\quad]{\quad\varphi\quad}
    \bU
    ,\quad
    % \bigcup_{s\in \bG}s\cdot \varphi(\bK)\cdot s^{-1}\subset V\cap S
    \varphi(\bK)\subset (\bG\cdot U)\cap S
  \end{equation*}
  the morphism (embedding) $\varphi$ is $\varepsilon$-close to a morphism (respectively embedding) $\bK\to \bG$.
\end{theorem}

\begin{remark}\label{re:cpctbetter}
  For compact $\bG$ (the case \cite{2212.06255v1} overwhelmingly focused on) the statement can be streamlined: no reference to $S$ need be made at all, since sufficiently small neighborhoods of compact subsets of $\bU$ are of course already ad-bounded. 
\end{remark}

% % OLD: VERSION WHEREBY $\varphi(\bK)$ IS CONTAINED IN ARBITRARY $\bG$-CONJUGATES OF $V$.
% % 
% % \begin{remark}\label{re:cpctbetter}
% %   The slight awkwardness in the phrasing \Cref{th:mor.almost2gp} is due to the fact that for non-compact $\bG$ it is possible that {\it no} neighborhood $V\supset \bG$ (or indeed, origin neighborhood) have {\it all} of its conjugates contained in a pre-selected small neighborhood of $\bG$.
% %  
% %   For compact $\bG$ (which was the case \cite{2212.06255v1} overwhelmingly focused on) the statement can be streamlined into simpler language (albeit with identical substantive content): there is an origin neighborhood $U\ni 1\in \bG$ so that every $\bK\xrightarrow{\varphi}\bU$ with $\varphi(\bK)\subset (\bG\cdot U)\cap S$ etc. Indeed, $\bG$ being compact, $1\in \bU$ has a local neighborhood system consisting of $\Ad(\bG)$-invariant sets. 
% % \end{remark}

\pf{th:mor.almost2gp}
\begin{th:mor.almost2gp}
  \begin{enumerate}[(I),wide]
  \item {\bf : Plain-morphism version.} We will apply \Cref{th:almostmor} to $\bK$ and $\bG$ in place of the $\bK$ and $\bU$ of the earlier statement; the idea is to argue that the present hypothesis ensures, assuming $V$ small enough, the existence of
    \begin{equation}\label{eq:psi.close2phi}
      \bK
      \xrightarrow[\quad\text{continuous $\delta$-morphism}\quad]{\quad\psi\quad}
      \bG
      \quad
      \text{$\frac{\varepsilon}{2}$-close to $\varphi$}
    \end{equation}
    \Cref{th:almostmor} then further deforms {\it that} into a true morphism $\bK\to \bG$ $\frac{\varepsilon}{2}$-close to it, finishing the proof.

    To produce \Cref{eq:psi.close2phi}, I claim it is enough to select {\it any} continuous map $\bK\to \bG$ sufficiently uniformly close to the original $\varphi$: the $\delta$-morphism property will follow from that closeness and the hypothesis. Indeed, suppose
    \begin{equation*}
      \psi(s) = e^{\alpha(s)}\cdot \varphi(s)
      ,\quad
      s\in \bK
    \end{equation*}
    for $\bK\xrightarrow{\alpha}\fu$ taking sufficiently small(-norm) values. We then have
    \begin{equation*}
      \begin{aligned}
        \psi(s)\psi(t)
        &=
          e^{\alpha(s)}\cdot e^{\Ad_{\varphi(s)}(\alpha(t))}\cdot \varphi(s)\varphi(t)\\
        &=e^{\alpha(s)}\cdot e^{\Ad_{\varphi(s)}(\alpha(t))}\cdot \varphi(st)
          \quad\text{($\varphi$ being a morphism) and}\\
        \psi(st)
        &=
          e^{\alpha(st)}\cdot \varphi(st).
      \end{aligned}
    \end{equation*}
    These will be uniformly close, given the assumed boundedness of $\left\{\Ad_{\varphi(x)}\right\}_{s\in \bK}$.

    The problem, then, has been reduced to the following claim (in which ad-boundedness no longer plays a role): if $U$ is a sufficiently small identity neighborhood in $\bU$, morphisms $\bK\to \bU$ taking values in $\bG\cdot U$ are $\varepsilon$-close to continuous maps $\bK\to \bG$. To conclude, simply observe that there is a {\it tubular neighborhood} \cite[\S IV.5]{lang-fund} of $\bG\le \bU$ of the form
    \begin{equation*}
      \bG\cdot U = \bG\cdot \exp(W)\cong \bG\times W
    \end{equation*}
    for a neighborhood $W\ni 0$ in a summand of $\fu$ complementing $Lie(\bG)\le \fu$. Naturally, continuous maps $\bK\to \bG\times W$ with uniformly small $W$-component are uniformly close to maps into $\bG\cong \bG\times \{0\}\subset \bG\times W$.

  \item {\bf : Embeddings.} Like all Banach Lie groups, $\bU$ has {\it no small subgroups} \cite[Theorem III.2.3]{neeb-lc}: there is an origin neighborhood $W\ni 1\in \bU$ containing no non-trivial subgroups. If $\varepsilon>0$ is sufficiently small, mutually $\varepsilon$-close morphisms $\bK\to \bU$ will simultaneously send subgroups of $\bK$ to $W$ (or not), so will have equal kernels.  \qedhere
  \end{enumerate}  
\end{th:mor.almost2gp}

As mentioned in the Introduction, the preceding remarks were initially partly motivated by considerations on and around variants of Jordan's theorem.  We are concerned here with ``relative'' versions of the invariant \Cref{eq:jg}, pertinent to embeddings of topological groups.

% % Recall also \cite[Definition 2]{MR3830471} that {\it bounded} groups are those with finite upper bounds on the orders of their finite subgroups.
% %

\begin{definition}\label{def:reljg}
  Let $\bG\le \bU$ be a subgroup of a topological group.
  \begin{enumerate}[(1),wide]
  \item For a neighborhood $V\supset \bG$ of $\bG$ in $\bU$
    \begin{equation*}
      J_{\bG\subset V}
      :=
      \sup_{\substack{\bH\le \bU\\|\bH|<\infty\\\bH\subset V}}
      \inf_{\substack{\bA\trianglelefteq \bH\\\bA\text{ abelian}}}
      |\bH/\bA|
    \end{equation*}
    is the {\it $(\bG\subset V)$-relative (or approximate) Jordan constant}. 

  \item Similarly, the {\it $(\bG\le \bU)$-relative Jordan constant} (attached to the embedding) is
    \begin{equation*}
      J_{\bG\le \bU}
      :=
      \inf_{\substack{\text{nbhd }V\supset \bG}}J_{\bG\subset V}
    \end{equation*}

  \item More generally, for any class $\cO$ of neighborhoods of $\bG$ in $\bU$,
    \begin{equation*}
      J^{\cO}_{\bG\le \bU}
      :=
      \inf_{\substack{V\in \cO}}J_{\bG\subset V}
    \end{equation*}
    
  \item A variant, decorated with an additional `$u$' subscript for `uniform', is the {\it $(\bG\le \bU)$-relative uniform Jordan constant}:
    \begin{equation*}
      J^u_{\bG\le \bU}      
      :=
      \inf_{\substack{\text{nbhd }U\ni 1}}J_{\bG\subset \bG\cdot U}
      =
      J^{\left\{\bG\cdot U\right\}_U}_{\bG\le \bU}
    \end{equation*}
    The notion has an obvious right-handed counterpart involving $U\cdot \bG$ instead. 
  \end{enumerate}
  The subgroup (or embedding) is {\it $(\bullet,\square)$-relatively Jordan} (or just plain {\it relatively Jordan} when the context is understood) if the respective constant $J^{\square}_{\bullet}$ is finite. 
\end{definition}

\begin{remark}\label{re:obviousineq}
  Naturally,
  \begin{equation*}
    J_{\bG\subset V}\ge J_{\bG}
    ,\quad\forall\text{ nbhd }V\supset \bG
    \quad\text{and hence}\quad    
    J_{\bG\le \bU}
    \ge J_{\bG}    
  \end{equation*}
  in full generality: $J_{\bG\subset V}$ takes a supremum over a larger collection of finite groups. Also note that, the infimum defining $J^{\cO}_{\bG\le U}$ being narrower the smaller $\cO$ is, the invariant increases (weakly) with decreasing $\cO$ (under inclusion). 
\end{remark}

In this language, the approximate Jordan-type result in \cite[Theorem 2.19]{2212.06255v1} reads: 
\begin{equation}\label{eq:th219.summ}
  \forall\text{embedding}
  \quad
  \bG
  \text{ (compact)}
  \le
  \bU
  \text{ (Banach Lie)}
  \quad
  \text{is relatively Jordan}.
\end{equation}
That proof relies on Turing's work \cite{MR1503391} on {\it (finitely) approximable} Lie groups: those $\bG$ admitting, for every $\varepsilon>0$, $\varepsilon$-morphisms
\begin{equation*}
  \bK
  \xrightarrow{\quad\varphi\quad}
  \bG
  ,\quad
  \bG=\text{$\varepsilon$-neighborhood of the image }\varphi(\bK).
\end{equation*}
\cite[Theorem 2]{MR1503391} asserts (essentially, in slightly different phrasing) that all such $\bG$ have compact abelian identity component (see also \cite[Theorems 6 and 7]{zbMATH06988579}, which recover that result); the proof of \cite[Theorem 2.19]{2212.06255v1} relies on (the proof of) the auxiliary result \cite[Theorem 1]{MR1503391}, asserting that an approximable Lie group with a $d$-dimensional linear representation is also approximable by finite groups with the same property. A sketch of the logical chain, then, would be as follows:
\begin{itemize}[wide]
\item \cite[Theorem 2]{MR1503391} relies on the classical Jordan theorem, as well as \cite[Theorem 1]{MR1503391};
\item the latter's proof is in turn adapted in \cite[Theorem 2.19]{2212.06255v1}, obtaining the relative/approximate Jordan result \Cref{eq:th219.summ}. 
\end{itemize}
The preceding material now affords some rearrangement: \Cref{eq:th219.summ} need not rely on either \cite[Theorem 1 or 2]{MR1503391} (or repurposed proofs thereof), as all of these follow from \Cref{th:almostmor}. We record, for instance, the following self-evident consequence of that earlier result. 

The $\varepsilon$-closeness referred to in the statement, between subsets of a Banach Lie group, is assessed in the {\it Hausdorff metric} (\cite[\S 45, Exercise 7]{mnk}, \cite[Definition 7.3.1]{bbi}) on the collection of closed subsets of a metric space $(X,d)$ (here $\bU$ equipped with a complete metric topologizing it):
\begin{equation*}
  \begin{aligned}
    d_{\cat{Haus}}(A,B)
    &:=
      \inf\left\{\varepsilon>0\ |\
      A\subseteq B_{\varepsilon}
      \quad\text{and}\quad
      B\subseteq A_{\varepsilon}
      \right\}
    \\
    \bullet_{\varepsilon}
    &:=
      \text{closed $\varepsilon$-neighborhood} \{x\ |\ d(x,\bullet)\le \varepsilon\}.
  \end{aligned}    
\end{equation*}

\begin{corollary}\label{cor:haveembquot}
  Let $\bU$ be a Banach Lie group and $S\subseteq \bU$ an ad-bounded subset. There is a function
  \begin{equation*}
    \text{some interval }(0,a)
    \xrightarrow[\lim_0\delta=0]{\quad\delta\quad}
    (0,a)
  \end{equation*}  
  so that the image of every compact-domain $\delta(\varepsilon)$-morphism $\bK\xrightarrow{\varphi}\bU$ valued in $S\subset \bU$ is $\varepsilon$-close to a compact (hence Lie) subgroup of $\bU$ isomorphic to a quotient of $\bK$.  \qedhere
\end{corollary}

\begin{remark}\label{re:recearlier}
  \Cref{cor:haveembquot} does indeed recover \cite[Theorem 1]{MR1503391}: the image of a $\delta$-morphism $\bK\to \bU$ with finite (more generally, compact) domain is $\varepsilon$-close to a subgroup of $\bG$, which subgroup of course inherits the representations of $\bG$ by restriction.
\end{remark}

As to Jordan-type results and properties, note the following consequence (a {\it $\delta$-subgroup} is a subset with an abstract group structure so that the embedding is a $\delta$-morphism). 

\begin{corollary}\label{cor:almostjordbdd}
  For a Banach Lie group $\bU$, ad-bounded subset $S\subseteq \bU$ and sufficiently small $\delta$ a finite $\delta$-subgroup of $S$ is $\varepsilon$-close to a genuine subgroup.  \qedhere
\end{corollary}

Such a subgroup will of course have an abelian normal subgroup of index $\le \bJ_{\bU}$, which is finite \cite[Theorem 2]{MR3830471} as soon as $\bU$ is finite-dimensional with bounded component group. Since furthermore
\begin{itemize}[wide]
\item for compact $\bU$ we may take $S=\bU$;
\item and finitely approximable Lie groups are indeed compact (by \cite[Theorem 39.9]{wil_top}, being totally bounded by hypothesis and completely metrizable as all Lie groups are \cite[\S III.1.1, Proposition 1]{bourb_lie_1-3}),
\end{itemize}
\cite[Theorem 2]{MR1503391} follows. 

Turning to the setup of \Cref{th:mor.almost2gp}, we have the following strengthening (via the specialization \Cref{cor:th.2.19.cpct}) of \cite[Theorem 2.19]{2212.06255v1}.

\begin{corollary}\label{cor:th.2.19.noncpct}
  Let $\bG\le \bU$ be a Lie subgroup of a Banach Lie group and $S\subseteq \bU$ an ad-bounded subset.

  For some origin neighborhood $U\ni 1\in \bU$ we have
  \begin{equation}\label{eq:cor:th.2.19.noncpct:sttmt}
    \sup_{\substack{\bH\le \bU\\|\bH|<\infty\\\bH\subset \bG\cdot U\cap S}}
    \inf_{\substack{\bA\trianglelefteq \bH\\\bA\text{ abelian}}}
    |\bH/\bA|
    \le
    J_{\bG}.
  \end{equation}
\end{corollary}
\begin{proof}
  Indeed, by \Cref{th:mor.almost2gp} sufficiently small $U$ ensures that the finite groups $\bH$ of \Cref{eq:cor:th.2.19.noncpct:sttmt} admit embeddings into $\bG$, close to the original embeddings in $\bU$. 
\end{proof}

In particular, when $\bG\le \bU$ some neighborhood $\bG\cdot U$ will itself be ad-bounded, whence the following more precise version of \Cref{eq:th219.summ}. 

\begin{corollary}\label{cor:th.2.19.cpct}
  For a compact subgroup $\bG\le \bU$ of a Banach Lie group we have
  \begin{equation*}
    J_{\bG\le \bU} = J_{\bG}
  \end{equation*}
  in the notation of \Cref{def:reljg}.  \qedhere
\end{corollary}

\addcontentsline{toc}{section}{References}
%\bibliography{bib}{}
%\bibliographystyle{plain}

\Addresses

\end{document}